\documentclass[12pt,reqno]{amsart}
\usepackage{etoolbox}

\DeclareMathAlphabet\EuScript{U}{eus}{m}{n}
\DeclareMathAlphabet\EuScriptBold{U}{eus}{b}{n}
\DeclareMathAlphabet\Eurm{U}{eur}{m}{n}
\DeclareMathAlphabet\Eurb{U}{eur}{b}{n}
\newcommand{\mathcalb}{\EuScriptBold}

\usepackage{mathtools}
\usepackage{etoolbox}
\usepackage{url}

\DeclareFontFamily{U}{matha}{\hyphenchar\font45}
\DeclareFontShape{U}{matha}{m}{n}{
	<-6> matha5 <6-7> matha6 <7-8> matha7
	<8-9> matha8 <9-10> matha9
	<10-12> matha10 <12-> matha12
}{}
\DeclareSymbolFont{matha}{U}{matha}{m}{n}

\DeclareFontFamily{U}{mathx}{\hyphenchar\font45}
\DeclareFontShape{U}{mathx}{m}{n}{
	<-6> mathx5 <6-7> mathx6 <7-8> mathx7
	<8-9> mathx8 <9-10> mathx9
	<10-12> mathx10 <12-> mathx12
}{}
\DeclareSymbolFont{mathx}{U}{mathx}{m}{n}

\DeclareMathDelimiter{\vvvert} {0}{matha}{"7E}{mathx}{"17}%

\DeclarePairedDelimiterX{\normiii}[1]
{\vvvert}
{\vvvert}
{\ifblank{#1}{\:\cdot\:}{#1}}
\newcommand\scal[2]{\left\langle #1,#2\right\rangle}

\let\CMcal=\mathcal

\def\HH{\mathcal{H}}
\def\BH{\mathcalb{B}(\mathcal{H})}

\def\KH{\mathcalb{K}(\HH)}

\def\M{\mathfrak{M}}

\def\N{\mathbb{N}}

\def\C{\mathbb{C}}

\newcommand\ccphi{ {\mathcalb C}_{\Phi}({\mathcal H})}

\newcommand{\cci}{ {\mathcalb  C}_\iii({\mathcal H}) }

\newcommand{\GG}{\mathbb{G}}

\newcommand{\tr}{\operatorname{tr}}

\newcommand{\phip}{{\Phi^{(p)}}}

\newcommand{\AAA}{ {\mathscr A} }
\newcommand{\AAt}{ {\mathscr A}_t }
\newcommand{\BBB}{ {\mathscr B} }

\newcommand{\WWW}{\Omega}

\newcommand{\iii}{\infty}

\newcommand{\sumn}{\sum_{n=1}^\iii}

\def\BB{\mathfrak B}

\newcommand{\cphi}{\mathfrak{c}_{\Phi}}

\def\cphin{\cphi^{(\circ)}}
\def\cphip{\mathfrak{c}_{\Phi^{(p)}}}
\def\cphipn{\mathfrak{c}^{(\circ)}_{\Phi^{(p)}}}

\newcommand{\ccc}{ {\mathcalb C}}

\newcommand{\ccphip}{ {\mathcalb C}_{\Phi^{(p)}}({\mathcal H}) }

\newcommand{\ccphipn}{ {\mathcalb C}^{(\circ)}_{\Phi^{(p)}}({\mathcal H}) }

\newcommand{\ccphiz}{{\mathcalb C}_{\Phi^{*}}({\mathcal H})}

\newcommand{\ccphin}{{\mathcalb C}^{(\circ)}_{\Phi}({\mathcal H})}

\newcommand{\ccphipz}{{\mathcalb C}_{\Phi^{(p)^*}}({\mathcal H})}

\newcommand{\ccj}{ {\mathcalb  C}_1({\CMcal H}) }

\newcommand{\PP}{\mathbb{P}}

\usepackage{mathtools} 
\usepackage{fixdif} 
\usepackage{leftindex}

\makeatletter
\patchcmd\maketitle
{\uppercasenonmath\shorttitle}
{}
{}{}
\patchcmd\maketitle
{\@nx\MakeUppercase{\the\toks@}}
{\the\toks@}
{}
{}{}
\patchcmd\@settitle{\uppercasenonmath\@title}{\Large}{}{}
\patchcmd\@setauthors
{\MakeUppercase{\authors}}
{\authors}
{}{}
\makeatother

\usepackage{amsmath,amssymb,amsthm, amsfonts}
\usepackage{color}
\usepackage{url}
\usepackage{tikz-cd}
\usepackage{combelow}
\input{mathrsfs.sty}
\usepackage[utf8]{inputenc}
\usepackage[T1]{fontenc}
\usepackage{enumerate}
\newtheorem{theorem}{Theorem}[section]

\newtheorem{definition}{Definition}[section]

\newtheorem{corollary}{Corollary}[section]
\newtheorem{proposition}{Proposition}[section]

\newtheorem{lemma}{Lemma}[section]
\newtheorem{remark}{Remark}[section]
\newtheorem{example}{Example}[section]
\numberwithin{equation}{section}

\newcommand{\nizcphiz}{\mathfrak{c}_{\Phi^*}}

\newcommand{\Dj}{\mbox{\rule[.75ex]{.4em}{.7pt}\kern-.4em D}}
\newcommand{\djm}{\mbox{\kern.1em\rule[1.2ex]{.3em}{.7pt}\kern-.4em d}}

\newcommand{\nuBp}{\nu_{|\BBB|^p}}

\usepackage[colorlinks=true]{hyperref}
\hypersetup{urlcolor=blue, citecolor=red , linkcolor= blue}

\begin{document}
		\title[Pettis integrability of functions with values in separable symmetrically-normed ideals...]{Pettis integrability of functions with values in separable symmetrically-normed ideals and related norm estimates}
		

		\author[Mihailo Krsti\' c]{Mihailo Krsti\' c}
		\address{Faculty of Mathematics, University of Belgrade, Studentski trg 16, Belgrade, Serbia
		}
		\email{\url{mihailo.krstic@matf.bg.ac.rs}}
        
		\author[Matija Milovi\' c]{Matija Milovi\' c}
		\address{Faculty of Mathematics, University of Belgrade, Studentski trg 16, Belgrade, Serbia
		}
	
	\email{\url{matija.milovic@matf.bg.ac.rs}}

    \author[Stefan Milo\v sevi\' c]{Stefan Milo\v sevi\' c}
		\address{Faculty of Mathematics, University of Belgrade, Studentski trg 16, Belgrade, Serbia
		}
	
	\email{\url{stefan.milosevic@matf.bg.ac.rs}}


		\date{\today}

		\begin{abstract} 

        For a weakly$^*$ integrable function $\AAA\colon\WWW\to\BH$ taking values in the positive cone of $\ccphin$, we prove that $\AAA$ is $\ccphin$-Pettis integrable if and only if its weak$^*$ integral over $\WWW$ belongs to $\ccphin$, in which case its Pettis norm is equal to $\|\int_\WWW\AAA\,d\mu\|_\Phi$.
         In addition, we prove the  inequality  $$\left\|\int_E\AAA^*\BBB\,d\mu \right\|_{\Phi^{(p)}} \leqslant \|\AAA\|_{L^q_s}\cdot\left\|\sqrt[p]{\int_E|\BBB|^p\,d\mu}\right\|_{\Phi^{(p)}}, $$
        where $\Phi^{(p)}$ is the $p$-modification of the function $\Phi$ and  the functions $\AAA$ and $\BBB$ belong to the suitable spaces of operator-valued functions. Finally, we provide similar estimates of the Pettis norm of $\AAA^*\BBB$.
		\end{abstract}
		
		\maketitle

\section{Introduction}

\subsection{Ideals of compact operators}

Throughout the paper, $\KH$, $\cci$ and $\BH$ will denote the spaces of all finite-rank operators, compact linear and bounded linear operators, respectively, acting on an infinite-dimensional separable complex Hilbert space $\HH$. Furthermore, by $\mathfrak{c}_0$ we will denote the Banach space of all complex sequences converging to zero, by $\ell^p$ we will denote the spaces  of $p$-summable sequences for $1\leqslant p<+\iii$, while $\ell^\iii$ denotes the space of all bounded complex sequences.
By $\Phi$ we will denote a \textit{symmetric norming} (s.\,n.) function (see \cite[Section~3, Chapter~3]{GK}). Associated with $\Phi$, we will define
\[
\cphi=
\left\{
a\in\mathfrak{c}_0
\,\,\Big|\,\,
\Phi(a)=
\lim_{n\to\infty}
\Phi(a_1,a_2,\ldots,a_n,0,0,\ldots)
<+\iii
\right\}.
\]
The space $\cphi$, endowed with the norm $\Phi$, is complete and all standard properties of $\Phi$ are naturally extended to it. In particular, the norms defining the spaces $\ell^p$ are examples of s.\,n. functions and will also be denoted by $\ell^p$. Note that s.\,n. functions $\ell^1$ and $\ell^\iii$ are extremal, i.\,e. for every s.\,n. function $\Phi$ holds $\ell^\iii(a)\leqslant\Phi(a)\leqslant\ell^1(a)$ for any complex  sequence $a$.
Two s.\,n. functions $\Phi$ and $\Psi$ are called \textit{equivalent}, denoted by $\Phi\sim\Psi$, whenever there exist constants $c,C>0$ such that
$
c\cdot\Phi(a)\leqslant\Psi(a)\leqslant C\cdot\Phi(a)
$ for every $a\in\cphi$. Due to \cite[Theorem~1.16, (g)]{S} we have $\Phi\sim\Psi$ if and only if $\cphi=\mathfrak{c}_\Psi$.

Every s.\,n. function gives rise to a \textit{symmetrically-normed} (s.\,n.) ideal of compact operators (for more details see \cite[Chapter III, Section 4]{GK}) that will be denoted by $\ccphi$. The corresponding \textit{symmetric norm} is given by
$
\|A\|_\Phi=
\Phi\bigl((s_n(A))_{n=1}^{\infty}\bigr)$ for every $A\in\ccc_\Phi(\HH)$,
where
$
s_1(A)\geqslant s_2(A)\geqslant\cdots
$
are the \textit{singular values} of the compact operator $A$. If $0\leqslant A\leqslant B$ are such that $B$ belongs to some s.\,n. ideal then the operator $A$ belongs to the same ideal, with $s_n(A)\leqslant s_n(B)$ for all $n\in\N$, implying that $\|A\|_\Phi\leqslant\|B\|_\Phi$ for every s.\,n. function $\Phi$ (see \cite[Lemma~II.1.1]{GK} and \cite[III.2.2 on Page 69]{GK}). In the case where $\Phi\sim\ell^\iii$, the theory of singular values of bounded operators allows us to define $\|\cdot\|_\Phi$ for every operator in $\BH$. The space of all bounded operators on $\HH$ endowed with this norm is denoted by $\mathcalb{B}_\Phi(\HH)$. By $\ccphin$ we will denote the closure of $\KH$ within $\ccphi$.  By \cite[Corollary~III.6.1]{GK}, the ideal $\ccphi$ is separable if and only if
$
\ccphi=\ccphin
$. Moreover, due to \cite[Theorem~III.6.2]{GK} every separable s.\,n. ideal is of this form.

For an arbitrary s.\,n. function $\Phi$ and $1\leqslant p<+\infty$, we define its \textit{$p$-modification} by
$
\Phi^{(p)}((z_n)_{n=1}^{\iii})
=
\sqrt[p]{
\Phi\bigl((|z_n|^p)_{n=1}^{\iii}\bigr)
}$, which is an s.\,n. function as well (see \cite[Page~2064]{J09ii}).
The corresponding sequence space $\cphip$ consists of all sequences of complex numbers
$
z=(z_n)_{n=1}^{\iii}
$
such that
$
(|z_n|^p)_{n=1}^{\iii}\in\mathfrak{c}_\Phi.
$
 We note that the norms $\ell^p$ are precisely the $p$-modifications of $\ell^1$. The corresponding operator norm is given by
$\|A\|_{\Phi^{(p)}} = \sqrt[p]{\big\||A|^p\big\|_\Phi}$,
for every $A\in\BH$ such that $|A|^p\in\ccphi$, where $|A|=\sqrt{A^*A}$ is \textit{absolute value} of $A$. In particular, for $1\leqslant p<+\infty$, the \textit{Schatten trace norms}
$
\|A\|_p=
\sqrt[p]{
\sum_{n=1}^{\infty}
s_n^p(A)
}$, defined for
$A\in\ccc_p(\HH)$,
coincide with the $p$-modifications of the trace norm $\|\cdot\|_1$.
We note that for $p>1$ holds $\ccc_1(\HH)\subsetneq\ccc_p(\HH)\subseteq\ccphipn\subseteq\ccphip$, implying that $\ccphipn$ is a strictly larger ideal than $\ccj$.

Next, by $\BB$ will be denoted the set of all orthonormal systems in $\HH.$ Also, for $f,g\in\HH,$ we will denote by $f\otimes g^*$ the bounded linear operator (of rank 1) satisfying $(f\otimes g^*)h=\langle h, g\rangle f$ for every $h\in\HH$. Then, we have equalities $\|f\otimes g^*\|=\|f\|\cdot\|g\|$, $\textup{tr}(f\otimes g^*)=\scal{f}{g}$ and $|f\otimes g^*|
=
\frac{\|f\|}{\|g\|}\,g\otimes g^*$, for $g\neq 0$. Also, for $f,g,h,k\in\HH$ holds $(f\otimes g^*)(h\otimes k^*)=\scal{h}{g}(f\otimes k^*)$. Let us recall (see \cite[Section 2, Chapter II]{GK}) that for every $A\in\cci$ there exist elements  $e,f\in\BB$ such that $A=\sum_{n=1}^{\iii}s_n(A)\,e_n\otimes f_n^*$ (we will refer to this sum as the \textit{Schmidt expansion} of the compact operator $A$). The o.\,n.\,s. $(e_n)_{n=1}^\iii$ is a basis for the closure of the range of $A$ i.\,e. for $\overline{\mathcal{R}(A)}$, while o.\,n.\,s. $(f_n)_{n=1}^\iii$ is a basis for orthogonal complement of kernel of $A$, i.\,e. for $\mathcal{N}(A)^\perp=\overline{\mathcal{R}(A^*)}=\overline{\mathcal{R}(|A|)}$. 

The following result provides a useful characterization for an operator $A\in\BH$ to belong to the ideal $\ccphi$.

\begin{theorem}\cite[Prop. 2.6]{S}\label{Simon}
For every $A\in\BH$ we have $A\in\ccphi$ if and only if for every $e,f\in\BB$ holds $(\langle Ae_n,f_n\rangle)_{n=1}^\iii\in\cphi$, in which case 
\begin{equation}\label{normaphi}
\|A\|_{\Phi}=\sup_{e,f\in\BB}\Phi((\langle Ae_n,f_n\rangle)_{n=1}^\iii),
\end{equation}
with the supremum being attained on the systems of vectors in the Schmidt expansion of $A$.
\end{theorem}

For a given s.\,n. function $\Phi$, we will denote by $\Phi^*$ its adjoint s.\,n. function (see \cite[Section~11, Chapter III]{GK}). The following \textit{Abstract Hölder inequality} holds
\begin{equation*}\label{abstraktnihelder}
    \sumn |a_nb_n|\leqslant\Phi(a)\cdot\Phi^*(b),\qquad a\in\cphi,\,\, b\in\nizcphiz.
\end{equation*}
This inequality is essential for the following duality theorem (see \cite[Theorem~3.2]{S}). For the convenience of the reader, we provide its full statement.

\begin{theorem}\label{duali}

 Let $\Phi$ and $\Phi^*$ be mutually adjoint s.n. functions. Then, for every $B\in\ccphiz,$ the map $A\mapsto\tr(AB)$ is a bounded linear functional on $\ccphi$ with the norm $\|B\|_{\Phi^*}.$ In the case where $\Phi\nsim\ell^1$, every bounded linear functional on $\ccphin$ is of this form, i.\,e. $(\ccphin)^*\cong\ccphiz$ (isometrically). If $\Phi\sim\ell^1,$ we have $(\ccphi)^*\cong\mathcalb{B}_{\Phi^*}(\mathcal{H})$ in the same manner.
\end{theorem}

Note that for every s.\,n. function $\Phi$, the dual space of $\ccphin$ can be seen as a certain space of operators ($\ccphiz$ or $\mathcalb{B}_{\Phi^*}(\mathcal{H})$). Thus, we will use notation $(\ccphin)^*$ when we do not need to separate these two cases. 

In the case where $\Phi\nsim\ell^1$, for every $A\in\ccphi$ and $X\in\ccphiz$ we have  $AX\in\ccj$ and therefore $\tr(AX) = \sum_{m=1}^\infty \scal{AX g_m}{g_m}$ for an arbitrary o.\,n.\,b. $(g_m)_{m=1}^\iii$. Let $X=\sumn s_n(X)\,e_n\otimes f_n^*$ be its Schmidt expansion. By completing (if necessary) $(f_n)_{n=1}^\iii$ to the full o.\,n.\,b. $(g_m)_{m=1}^\iii$ of $\HH$ with the o.\,n.\,b. of $\mathcal{N}(X)$, we get the following formula
\begin{equation}\label{trAX}
\tr(AX) = \sum_{m=1}^\infty \scal{AX g_m}{g_m} = \sum_{n=1}^\infty \scal{AX f_n}{f_n} = \sum_{n=1}^\iii s_n(X)\scal{A e_n}{f_n},
\end{equation}
that will be used several times throughout the paper.

The next lemma provides us with a variant of the Cauchy-Schwarz inequality, which we will also need in the sequel. It is based on Kato’s mixed Schwarz inequality $|\scal{Xf}{g}|^2 \leqslant \scal{|X|f}{f}\cdot\scal{|X^*|g}{g}$ (see \cite[inequality (2)]{Kit}).
\begin{lemma}
    Let $\Phi$ be an arbitrary s.\,n. function, and let $X\in(\ccphin)^*$. Then, for every  $0\leqslant A\in\ccphin$ holds
\begin{equation}\label{tragpoz}
|\tr(AX)|\leqslant \sqrt{\tr(A |X|)}\cdot\sqrt{\tr(A|X^{*}|)}.
\end{equation}
\end{lemma}
\begin{proof}
    As $|X|,|X^*|\in (\ccphin)^*$, we get that the operators $AX,A|X|$ and $A|X^*|$ are nuclear. Hence, for every o.\,n.\,b. $e$ holds
    \begin{multline*}
    |\tr(AX)|=\left|\tr\left(\sqrt{A}X\sqrt{A}\right)\right|=\left|\sumn\scal{\sqrt{A}X\sqrt{A}e_n}{e_n}\right|=\left|\sumn\scal{X\sqrt{A}e_n}{\sqrt{A}e_n}\right|\\\leqslant\sumn\left|\scal{X\sqrt{A}e_n}{\sqrt{A}e_n}\right|\leqslant\sumn\sqrt{\scal{|X|\sqrt{A}e_n}{\sqrt{A}e_n}}\cdot\sqrt{\scal{|X^*|\sqrt{A}e_n}{\sqrt{A}e_n}}\\\leqslant\sqrt{\sumn\scal{|X|\sqrt{A}e_n}{\sqrt{A}e_n}\cdot\sumn\scal{|X^*|\sqrt{A}e_n}{\sqrt{A}e_n}}=\sqrt{\tr(A|X|)\cdot\tr(A|X^*|)},
\end{multline*}
proving the inequality \eqref{tragpoz}.
\end{proof}

\subsection{Vector Measures and Integration in Banach spaces}

Throughout the paper, $(\Omega,\M,\mu)$ denotes a measurable space endowed with a complete measure $\mu$.  We will usually suppress $\M$ from the notation when there is no risk of ambiguity. In the special case $\Omega=\N$, the measure $\mu$ is assumed to be the counting measure on the power set $\mathcal{P}(\N)$ and will also be suppressed from the notation.

We next recall several standard notions and notation from the theory of vector measures; additional details can be found in the monograph \cite{DU}. 
A mapping $\nu\colon\M\to X$, where $X$ is a Banach space, is called a \textit{vector measure} provided 
that 
$
\nu(A\sqcup B)=\nu(A)+\nu(B)
$
for every pair of disjoint sets $A,B\in\M$.
By a straightforward induction argument, each vector measure is finitely additive. 
If, moreover, for every sequence $(E_n)_{n=1}^{\iii}$ of pairwise disjoint sets in $\M$ the equality
$
\nu\left(\bigsqcup_{n=1}^{\iii}E_n\right)=\sum_{n=1}^{\infty}\nu(E_n)$
holds in the norm topology of $X$, then $\nu$ is said to be \textit{countably additive}. Furthermore, we say that $\nu$ is \textit{$\mu$-continuous}, and write $\nu\ll\mu$, if for every $\varepsilon>0$ there exists $\delta>0$ such that for each $E\in\M$ satisfying $\mu(E)<\delta$, one has
$
\|\nu(E)\|_X<\varepsilon$.
It is immediate that every $\mu$-continuous vector measure vanishes on sets of $\mu$-measure zero.

A function $\varphi\colon\Omega\to X$ is called \textit{strongly $\mu$-measurable} if there exists a sequence $(s_n)_{n=1}^{\infty}$ of $\mu$-\textit{simple functions} satisfying $\lim_{n\to\infty}\|s_n(t)-\varphi(t)\|_X=0$ for $\mu$-almost all $t\in\Omega$. Here, a $\mu$-simple function is a function of the form $s=\sum_{k=1}^{m}x_k\chi_{A_k}$, where $(A_k)_{k=1}^{m}$ are pairwise disjoint sets in $\M$ of finite $\mu$-measure and $(x_k)_{k=1}^{m}$ is a finite family in $X$. In the scalar-valued case $X=\C$, we simply refer to such functions as $\mu$-measurable. It follows that if $\varphi$ is strongly $\mu$-measurable, then the mapping $t\mapsto\|\varphi(t)\|_X$ is $\mu$-measurable. 
A function $\varphi\colon\Omega\to X$ is called \textit{weakly $\mu$-measurable} whenever for every $x^\ast\in X^\ast$ the scalar-valued function $ t\mapsto\langle\varphi(t),x^\ast\rangle$ is $\mu$-measurable.  
If, in addition, these mappings belong to $L^1(\Omega, \mu)$, we will say that $\varphi$ is \textit{weakly integrable}. 
Moreover, if for every $E\in\M$ there exists $\varphi_E\in X$ such that $\scal{\varphi_E}{x^*}=\int_E\scal{\varphi(t)}{x^*}\,d\mu(t)$ for every $x^*\in X^*$ we will say that $\varphi$ is \textit{Pettis integrable}. The space of Pettis integrable functions will be denoted by $\PP(\Omega, \mu, X)$, with the norm given by $\|\varphi\|_{\mathbb{P}}=\sup_{\|x^*\|_{X^*}=1}\int_\Omega|\scal{\varphi(t)}{x^*}|\,d\mu(t)$. We will identify $\varphi_1$ and $\varphi_2$ in $\PP(\Omega, \mu, X)$  if for every $x^*\in X^*$ the equality $\scal{\varphi_1(t)}{x^*} = \scal{\varphi_2(t)}{x^*}$ holds for $\mu$-almost all $t\in\WWW$. If $X$ is separable, the previous identification reduces to $\varphi_1(t) = \varphi_2(t)$ for $\mu$-almost every $t\in\WWW$ (see \cite[Corollary~7, Page 48]{DU}). Finally, we note that for every $\varphi\in\PP(\Omega, \mu, X)$ holds $\|\varphi\|_{\mathbb{P}}\geqslant\|\varphi_E\|_X$ for all $E\in\M$.

 For $1\leqslant p<+\infty$, a strongly $\mu$-measurable function $\varphi\colon\Omega\to X$ is called \textit{Bochner $p$-integrable} if $\int_\Omega\|\varphi(t)\|_X^p\,d\mu(t)<+\infty$. The space of all such functions, identified 
 $\mu$-almost everywhere, is denoted by $L^p(\Omega,\mu,X)$. It is a Banach space  endowed with the natural norm $\|\varphi\|_{L^p}:=\sqrt[p]{\int_\Omega\|\varphi(t)\|_X^p\,d\mu(t)}$. The functions in $L^1(\Omega,\mu,X)$ are called \textit{Bochner integrable}. The \textit{Bochner integral} is defined for such functions (for the construction of the Bochner integral see \cite[Definition~1.2.1]{ABS} and \cite[Proposition~1.2.2]{ABS}) and 
the inequality $\left\|\varphi_E\right\|_X\leqslant\int_E\|\varphi(t)\|_X\,d\mu(t)$ holds for every $E\in\M$. Every Bochner integrable function is Pettis integrable, and the corresponding integrals coincide.

We will say that $\AAA\colon\WWW\to\BH$ is weakly$^*$ $\mu$-measurable (weakly$^*$ integrable) if for every $f\in\HH$ the scalar function $t\mapsto\scal{\AAt f}{f}$ is $\mu$-measurable (integrable). Under these assumptions, the scalar functions $t\to\tr(\AAt X)$ are $\mu$-measurable (integrable) for every $X\in\ccj$ (see \cite[Lemma~1.1]{MMS}). Such functions are also called \textit{Gelfand integrable}. Whenever $\AAA$ is Gelfand integrable, there exists an operator $\int_E\AAA\,d\mu\in\BH$ satisfying
$\scal{\left( \int_E \AAA d\mu \right) f}{g}= 
\int_{E} \scal{\AAA_tf}{g}d\mu(t)$ for all $ E\in\M$ and $f,g\in\HH$. For more information on weak$^*$ integration see \cite[Page~52-53]{DU} and introduction of \cite{J05}. The corresponding vector space (with the identification $\mu$-almost everywhere, see \cite[Lemma~2.14]{MAMK})  will be denoted by $ \GG(\WWW,\mu,\BH)$. It is a normed space with the norm defined by (see \cite[Proposition~1.8]{Matija})
\begin{equation*}\label{Gnorm}
\| \AAA \|_{\GG} =  \sup_{\| X \|_{\ccj} =1 } \int_\Omega | \tr (\AAA_t X) | \, d\mu(t) = \sup_{\| f \| = \| g \| =1 } \int_\Omega | \langle \AAA_t f, g \rangle | \, d\mu(t),
\end{equation*}
and it is not complete in the general case, see \cite[Theorem 1.4]{MK}. Throughout the paper, by $\int_E\AAA\,d\mu$ we denote the weak$^*$ operator-valued (o.\,v.) integral, unless stated otherwise.

Every weakly$^*$ integrable $\AAA\colon\WWW\to\BH$ induces a $\BH$-valued measure $\nu_\AAA\colon\M\to\BH$, defined by
$\nu_\AAA(E)=\int_E\AAA\,d\mu$ for all $E\in\M$.
Every such measure is weakly$^*$ countably additive and need not be countably additive in the operator norm, see \cite[Example 2.13]{MAMK} and \cite[Theorem 1, Page 10]{DU}. 
Moreover, if we have $\AAA_t\geqslant0$ for $\mu$-almost all $t\in\Omega$ then $0\leqslant\nu_\AAA(E)\leqslant\nu_\AAA(F)$ whenever $E,F\in\M$ and $E\subseteq F$. Specially, in this case the following lemma provides simple criterion for countable additivity of measure $\nu_\AAA$ (see \cite[Theorem~1.3 a)]{MMS}).
\begin{theorem}\label{prbradit}
Let $\Phi$ be an arbitrary s.\,n. function and let $\AAA\colon\WWW\to\BH$ be a weakly$^*$ integrable o.\,v. function such that $\AAA_t\geqslant0$ for a.e. $t\in\WWW$ and $\int_\WWW\AAA\, d\mu\in\ccphin$. Then the induced measure $\nu_\AAA\colon\M\to\ccphin$ is countably additive.
\end{theorem}

The following space was introduced in \cite{BJN} in the setting of Banach space ope\-rators and under the additional assumption that the corresponding norm is finite. For more details see the text following \cite[Proposition~3.1]{BJN}.

\begin{definition}
Let $1\leqslant p<+\iii$ and let $\AAA \colon \WWW \rightarrow \BH$ be a weakly$^*$ $\mu$-measurable function. 
We say $\AAA$ is strongly $p$-integrable if $\int_\WWW \|\AAA_t f \|^p\,d\mu(t) < +\iii$ for all $f \in \HH$, and we denote the set of all such functions by $L^p_s(\Omega, \mu, \BH)$, with the customary identification of $\mu$-almost everywhere equal functions.
\end{definition}

Note that, unlike in the general situation in \cite{BJN}, we do not need to assume that the mappings $t\mapsto\AAt f$ are $\mu$-measurable for every $f\in\HH$. Indeed, this follows from weak$^*$ $\mu$-measurability of $\AAA$ and separability of $\HH$, due to the well known Pettis Measurability Theorem, see \cite[Theorem 1.\,1.\,20]{ABS}. It follows that the mapping $t\mapsto\|\AAt f\|$ is $\mu$-measurable, for every $f\in\HH$. For a general Banach space setting, we refer the reader to \cite{BGB} where different types of measurability of o.v. functions are analyzed. In particular, the previous considerations are also a consequence of \cite[Proposition~2.3]{BGB}, \cite[Proposition~2.6]{BGB}, and \cite[Corollary~3.6]{BGB}.

The following result provides a natural way to norm spaces $L^p_s(\Omega, \mu,\BH)$, due to the Closed Graph Theorem.

\begin{theorem}
For $1\leqslant p<+\iii, L^p_s(\WWW,\mu,\BH)$ is a vector space. Moreover, if
$\AAA\in L^p_s(\WWW,\mu,\BH)$, then the expression
$\|\AAA\|_{L^p_s}:=\sup_{\|f\|=1}\sqrt[p]{\int_\WWW \|\AAA_t f\|^p\,d\mu(t)}$ is finite and defines a norm on the vector space $L^p_s(\WWW,\mu,\BH)$.
\end{theorem}

If $\WWW=\N$ and $\mu$ is the counting measure on $\M=\mathcal{P}(\N),$ the corresponding space of strongly $p$-integrable functions will be denoted by $\ell^p_s(\BH)$. These spaces will play an important role in norm estimates that include operator H\"older inequalities.

The main focus of this paper is to prove Pettis integrability of functions with values in certain  ideals of operators. In the case of the ideal of compact operators we have the following two theorems that are direct consequences of
\cite[Theorem 2.3]{MMS}.
\begin{theorem}\label{TPetisA}
Let $(\WWW,\M,\mu)$ be a finite measure space and $\AAA\colon\WWW\rightarrow\BH$ be a weakly$^*$ integrable o.\,v. function, such that $\AAA_t\in\cci$ for every $t\in\WWW$. If the induced vector measure $\nu_\AAA$ is $\mu$-continuous then $\AAA\in\PP(\WWW,\mu,\cci)$.
\end{theorem}
The second theorem exploits the positivity condition and does not rely on the finiteness assumption of measure $\mu$.
\begin{theorem}\label{TPetisB}
Let $\AAA\colon\WWW\rightarrow\BH$ be weakly$^*$ integrable o.\,v. function, such that $0\leqslant\AAA_t\in\cci$ for every $t\in\WWW$ and $\int_\WWW\AAA\,d\mu\in\cci$. Then, we have that $\AAA\in\PP(\WWW,\mu,\cci)$.
\end{theorem}

\section{Preliminary results on  integrability of $\ccphin$-valued functions}

In this section, we provide some preliminary results concerning o.\,v. functions that take values in separable s.\,n. ideals.
As in the case of s.\,n. ideals, 
we denote by $\mathfrak{c}^{(\circ)}_{\Phi}$ the closure of the space of all complex sequences with finitely many nonzero coordinates in $\mathfrak{c}_\Phi$. It can be shown that $a\in\cphin$ if and only if $\lim_{n\to\infty}\Phi((a_k)_{k=n}^\iii)=0$, implying that $a\in\cphipn$ if and only if $(|a_n|^p)_{n=1}^\iii\in\cphin$. Furthermore, by \cite[Theorem~2.7, (b)]{S} we have $A\in\ccphin$ if and only if  $(s_n(A))_{n=1}^\iii\in\cphin$.  Similarly to \cite[Remark~1.2]{Matija}, we have the following result.

\begin{proposition}\label{KriterijumPhiNula}
    Let $\Phi$ be an s.\,n. function and $A\in\BH$. Then, $A\in\ccphin$ if and only if $(\scal{Ae_n}{e_n})_{n=1}^\iii\in\cphi^{(\circ)}$ for every o.\,n.\,s. $(e_n)_{n=1}^\iii$ in $\HH$.
\end{proposition}
\begin{proof}
    Let $A\in\ccphin$ and $e\in\BB$ be arbitrary, and let $M:=\overline{\mathfrak{Lin}\{e_n:n\in\N\}}$. For all $m\in\N$ define orthogonal projection $P_m=P_M-\sum_{k=1}^me_k\otimes e_k^*$, where $P_M$ is orthogonal projection on the subspace $M$. Then
    $$\Phi((\scal{Ae_n}{e_n})_{n=m+1}^\iii)=\Phi((\scal{P_mAP_me_n}{e_n})_{n=1}^\iii)\leqslant\|P_mAP_m\|_\Phi,$$
    due to the equality \eqref{normaphi}. Moreover, $P_m\rightarrow0$ strongly, as $m\rightarrow\iii$. Thus, we have that $P_mAP_m\rightarrow0$ in the norm of separable ideal $\ccphin$, due to \cite[Theorem~\text{III}.\,6.\,3]{GK}, hence $(\scal{Ae_n}{e_n})_{n=1}^\iii\in\cphin$. 
    
Conversely, note that $(\scal{Ae_n}{e_n})_{n=1}^\iii\in\mathfrak{c}_0$ for every $e\in\BB$. The compactness of $A$ now follows from  \cite[Theorem~1.8.7]{Rin}. In the case where $\Phi\nsim\ell^\iii$,
as $$|\scal{A_{\mathfrak{R}}h}{h}|^2+|\scal{A_{\mathfrak{I}}h}{h}|^2=|\scal{Ah}{h}|^2,\qquad h\in\HH,$$for every $e\in\BB$ we get $(\scal{A_{\mathfrak{R}}e_n}{e_n})_{n=1}^\iii,(\scal{A_{\mathfrak{I}}e_n}{e_n})_{n=1}^\iii\in\cphi^{(\circ)}.$  By taking o.\,n.\,s. that consists of eigenvectors of $A_{\mathfrak{R}}$ and $A_{\mathfrak{I}}$ respectively, we get that the sequence of singular values of $A_{\mathfrak{R}}$ and $A_{\mathfrak{I}}$ are in $\cphin$, proving that $A_{\mathfrak{R}}$ and $A_{\mathfrak{I}}$ are in the ideal $\ccphin$. Thus, $A=A_{\mathfrak{R}}+iA_{\mathfrak{I}}\in\ccphin$, which proves the assertion.
\end{proof}

By Theorem \ref{duali}, the dual space of the ideal $\ccphin$ is described in a useful way, as the space of operators. Furthermore, as these ideals are separable, we are able to improve the measurability properties of weakly$^*$ $\mu$-measurable $\AAA\colon\WWW\to\ccphin$. More precisely, the following holds

\begin{lemma}\label{glavnaPomoc} Let $\Phi$ be an arbitrary s.\,n. function and $\AAA\colon\WWW\rightarrow\BH$ be a weakly$^*$ $\mu$-measurable, such that $\AAA_t\in\ccphin$ for all $t\in\Omega$. Then $\AAA$ is $\ccphin$-strongly $\mu$-measurable.
\end{lemma}
\begin{proof}
    Let $X\in(\ccphin)^*$ be an arbitrary operator and $(e_n)_{n=1}^\infty$ be an o.\,n.\,b. of $\HH$. First, weak$^*$ $\mu$-measurability of $\AAA$ implies that the functions $\scal{\AAt Xe_n}{e_n}$ are $\mu$-measurable, for all $n\in\N$. Furthermore, as $\AAt X\in\ccj$ for all $t\in\WWW$, we get $\tr(\AAt X)=\sumn\scal{\AAt Xe_n}{e_n}$, implying $\ccphin$-weak $\mu$-measurability of $\AAA$. Finally, since the ideal $\ccphin$ is separable, it follows by  Pettis Measurability Theorem that $\AAA$ is $\ccphin$-strongly $\mu$-measurable, as proclaimed.
\end{proof}

Now we start investigating  $\ccphin$-Pettis integrability. We will do so by proving that  weakly$^*$ integrable functions, with some additional assumptions, have their $\ccphin$-Pettis integral over all measurable subsets equal to the corresponding weak$^*$ integrals. First, we prove that the converse always holds.

\begin{lemma}\label{petisjegeljfand}
    Let $\AAA\in\PP(\WWW,\mu,\ccphin)$. Then, $\AAA$ is weakly$^*$ integrable and its Pettis integral over $E\in\M$ is equal to its weak$^*$ integral $\int_E\AAA\,d\mu$.
\end{lemma}
\begin{proof}
Due to the formula \eqref{trAX}, for every $t\in\Omega$ and $f\in\HH$ holds  $\tr(\AAt(f\otimes f^*))=\scal{\AAt f}{f}$. Since for every $f\in\HH$ holds $f\otimes f^*\in\KH\subset\ccphiz$, the previous formula gives us weak$^*$ $\mu$-measurability of $\AAA$. Moreover, as
$\int_\WWW|\scal{\AAt f}{f}|\,d\mu(t)=\int_\WWW|\tr(\AAt(f\otimes f^*))|\,d\mu(t)<+\iii,$
we get weak$^*$ integrability of $\AAA$. Now, let $E\in\M$ be arbitrary and let $I_E\in\ccphin$ be the Pettis integral of $\AAA$ over $E\in\M$. Then, for every $f\in\HH$ holds
\begin{multline*}
        \scal{I_Ef}{f}=\textup{tr}(I_E(f\otimes f^*))=\int_E\textup{tr}(\AAA_t(f\otimes f^*))\,d\mu(t)\\
        =\int_E\scal{\AAA_tf}{f}d\mu(t)=\scal{\left(\int_E \AAA\,d\mu\right) f}{f}, 
\end{multline*} 
implying that $I_E=\int_E\AAA\,d\mu$, completing the proof.
\end{proof}
In the following theorem, we give some sufficient conditions 
for $\ccphin$-Pettis integrability.

\begin{theorem}\label{boljaVerzija}
Let $\AAA\colon\WWW\to\BH$ be a weakly$^*$ integrable function and $\Phi$ be arbitrary s.\,n. function. Also, let $\AAt\in\ccc_{\Phi}^{(\circ)}(\HH)$ for all $t\in\WWW$ and $\int_E \AAA \,d\mu\in\ccphi$ for every $E\in\M$. Moreover, let the measure $\nu_\AAA\colon \M\to\ccphi$ be countably additive. Then $\AAA\in\PP(\WWW,\mu,\ccphin)$.
\end{theorem}

\begin{proof}
First, we note that $\AAA$ is strongly $\mu$-measurable $\ccphin$-valued function, due to Lemma \ref{glavnaPomoc}. Thus, by \cite[Proposition~1.1.15]{ABS} we can assume that $\mu$ is $\sigma$-finite on $\WWW$. Hence, there exists the sequence $(\WWW_n)_{n=1}^\infty$ of mutually disjoint sets in $\M$, such that $\WWW = \bigsqcup_{n=1}^\infty \WWW_n$ and $\mu(\WWW_n)<+\iii$ for all $n\in\N$. Also, let us denote $F_k = \{t\in\WWW: k-1\leqslant\|\AAt\|_\Phi < k\}$ for $k\in\mathbb{N}$. We note that each $F_k\in\M$ because $\AAA$ is strongly $\mu$-measurable and $\WWW = \bigsqcup_{k=1}^\infty F_k$. From the construction of the previous sets we get
$$\int_{\WWW_n\cap F_k} \|\AAt\|_\Phi\, d\mu(t) \leqslant k\cdot\mu(\WWW_n)<+\iii,\quad n,k\in\mathbb{N},$$
implying the $\ccphin$-Bochner integrability of $\AAA$ on $\WWW_n\cap F_k$ for every $n,k\in\mathbb{N}$. The rest of the proof is similar to \cite[Corollary~5, Page 70]{DU}, but without the assumption of finiteness of measure $\mu$. For ease of notation, we denote $S_n=\bigsqcup_{k,l=1}^n \WWW_l\cap F_k$ for $n\in\N$. Then, $S_n\uparrow\WWW$ as $n\to\iii$ and
$$\nu_\AAA (E\cap S_n) = \int_E \AAA \chi_{S_n} d\mu\in\ccphin,\qquad n\in\N,\,\,E\in\M.$$
Due to the assumed countable additivity of $\nu_\AAA$ and the fact that $\ccphin$ is closed in $\ccphi$, we get
$$\nu_\AAA(E) =\lim_{n\to\iii} \nu_\AAA (E\cap S_n)=\lim_{n\to\iii}\int_E \AAA \chi_{S_n} \,d\mu\in\ccphin,\qquad E\in\M.$$
Let $E\in\M$ and $y^*\in(\ccphin)^*$ be arbitrary. Due to $\ccphin$-Bochner integrability of $\AAA$ on each set $E\cap S_n$, we have
$$\scal{\nu_\AAA(E)}{y^*} =\lim_{n\to\iii}\scal{\int_{E\cap S_n}\!\!\!\AAA\,d\mu}{y^*}=\lim_{n\to\iii} \int_{E\cap S_n} \scal{\AAA_t}{y^*}d\mu(t).$$
Define the function $\Lambda_{\AAA,y^*}\colon\M\rightarrow\C$ by $\Lambda_{\AAA,y^*}(E)=\scal{\nu_\AAA(E)}{y^*}$ for $E\in\M$.
Then, $\Lambda_{\AAA,y^*}$ is a countably additive complex measure on the $\sigma$-algebra $\M$, thus having finite total variation $|\Lambda_{\AAA,y^*}|(\Omega)$. Since Bochner integration commutes with bounded linear functionals, we have the equality $\Lambda_{\AAA,y^*}(E)=\int_E\scal{\AAt}{y^*}d\mu(t)$  and $\int_{E}|\scal{\AAA_t}{y^*}|\,d\mu(t)\leqslant|\Lambda_{\AAA,y^*}|(\Omega)<+\infty$ for every $E\in\M$ such that $E\subset S_n$. Specially, we have $$\int_{\Omega}|\scal{\AAA_t}{y^*}|\cdot \chi_{S_n}(t)\,d\mu(t)\leqslant|\Lambda_{\AAA,y^*}|(\Omega)<+\infty,\qquad n\in\N,$$
and since the sequence $(S_n)_{n=1}^\infty$ increases to $\Omega$, by applying Monotone Convergence Theorem we obtain weak integrability of the function $\AAA$. Finally, due to Dominated Convergence Theorem, for every $E\in\M$ holds
\begin{equation*}
    \scal{\nu_\AAA(E)}{y^*}=\lim_{n\rightarrow\iii}\int_E\scal{\AAt}{y^*}\chi_{S_n}(t)\,d\mu(t)=\int_E\scal{\AAt}{y^*}d\mu(t),
\end{equation*}
showing that $\nu_\AAA(E)$ is the $\ccphin$-Pettis integral of $\AAA$, thus completing the proof.
\end{proof}
Note that even though we assumed that the weak$^*$ integrals $\int_E\AAA\,d\mu$ are in $\ccphi$ for all $E\in\M$, it turned out that they are actually in $\ccphin$.

Before moving on, we note that the previous theorem can be used to deduce $\ccphin$-Pettis integrability of functions of the form $\AAA X\BBB$ appearing in \cite[Theorem~3.1]{JKL20}. Indeed, countable additivity is already proven in the mentioned theorem, in the equivalent form $\lim_{n\to\infty}\|\int_{\delta_n} \AAA X\BBB\,d\mu - \int_\delta \AAA X\BBB \,d\mu\|_\Phi=0$ for $\delta_n\uparrow\delta$. All other assumptions are readily verified.

\section{$L^p_{\GG}(\WWW,\mu,\BH)$ spaces}

 For $\AAA\colon\WWW\to\BH$  we use the notations $\AAA^*$ and $|\AAA|$ for the mappings $t\mapsto\AAt^*$ and $t\mapsto|\AAt|$, respectively.  Also, for o.\,v. functions $\AAA$ and $\BBB$, we denote their pointwise product by $\AAA\BBB$. Furthermore, for $A\in\BH$, by $\sigma(A)$ we denote its spectrum. Also, for a normal operator $A$ and a function $\varphi$ that is continuous on $\sigma(A)$, we have defined $\varphi(A)\in\BH$ due to functional calculus.
Thus, whenever meaningful, we will use the notation $\varphi(\AAA)$ for the mapping $t\mapsto\varphi(\AAt)$. 

In \cite{J05}, the author introduced the space $L^2_{\GG}(\WWW,\mu,\BH)$. It is a space of all weakly$^*$ $\mu$-measurable $\BH$-valued functions, such that $|\AAA|^2$ is weakly$^*$ integrable. As $\scal{|\AAt|^2f}{f}=\|\AAt f\|^2$ for all $f\in\HH$ and $t\in\Omega$, this space coincides with the space $L^2_s(\WWW,\mu,\BH)$. We note that due to the previous equality, for every weakly$^*$ $\mu$-measurable function $\AAA$, the mapping $|\AAA|^2$ is also weakly$^*$ $\mu$-measurable. The same holds for $\AAA^*$ as $\scal{\AAt^*f}{f}=\overline{\scal{\AAt f}{f}}$ for all $f\in\HH$ and $t\in\Omega$. Also, due to Parseval's equality (see \cite[Section~3]{J05}), the o.\,v. function $\AAA\BBB$ is weakly$^*$ $\mu$-measurable, whenever $\AAA$
 and $\BBB$ are such. In this section, we introduce the generalization of this space for all $1\leqslant p<+\iii$. To do that, we first need to prove that for every weakly$^*$ $\mu$-measurable $\BH$-valued function $\AAA$, the o.\,v. function $|\AAA|^p$ is also weakly$^*$ $\mu$-measurable, for every $p>0$. 
 
 First, we prove a technical lemma considering weak$^*$ $\mu$-measurability of o.\,v. functions obtained due to the functional calculus. With $C(K)$ we will denote the set of all continuous functions $\varphi:K\to\C$ on the given compact set $K\subset\C$. 

 If $\AAA$ is weakly$^*$ $\mu$-measurable, then $\AAA^*$ is also such. Thus, for every $m,n\geqslant0$ such that at least one of them is non-zero, we get that the operators $\AAA^n\AAA^{*m}$ are also weakly$^*$ $\mu$-measurable. In the case where $m=n=0$ we have that for every unit $f\in\HH$, the function 
$$t\mapsto\scal{\AAA_t^0\AAA_t^{*0}f}{f}=\scal{If}{f}=1\text{ is $\mu$-measurable iff $(\WWW,\M,\mu)$ is $\sigma$-finite.}$$
Thus, for every polynomial function $z\mapsto p(z,\overline{z})$ such that $p(0,0) = 0$ the o.\,v. function $p(\AAA,\AAA^*)$ is weakly$^*$ $\mu$-measurable while in the $\sigma$-finite case the condition $p(0,0) = 0$ can be dropped.

\begin{lemma}\label{tv1}
    Let $\AAA\colon\WWW\rightarrow\BH$ be a weakly$^*$ $\mu$-measurable o.\,v. function such that all  $\AAA_t$ are normal operators. 
    Furthermore, assume that there is a compact set $K\subset\C$ such that $\sigma(\AAA_t)\subseteq K$ for all $t\in\Omega$ and let $\varphi\in C(K)$. Then the operator-valued function $\varphi(\AAA)\colon\WWW\to\BH$ is weakly$^*$  $\mu$-measurable provided at least one of the following two conditions is satisfied:
\begin{itemize}
    \item[(1)] $(\WWW,\mu)$ is $\sigma$-finite.
    \item[(2)] If $0\in K$ then $\varphi(0) = 0$.
\end{itemize}
\end{lemma}
\begin{proof}   
Stone-Weierstrass theorem gives us a sequence $(p_n)_{n=1}^\iii$ of polynomials such that the functions $z\mapsto p_n(z, \overline z)$ converge uniformly on the compact $K$ to the function $\varphi$. In the first case, all functions $p_n$ are $\mu$-measurable. In the second case, if $0\in K$ then $p_n(0,0)\to\varphi(0) = 0$ and therefore polynomials $p_n(z,\overline{z}) - p_n(0,0)$ also converge uniformly to $\varphi$ on $K$ and they are $\mu$-measurable due to remark prior to the lemma. Finally, if $0\notin K$ we can add $0$ to $K$ and extend $\varphi$ by $\varphi(0) = 0$ and reduce to the previous case. Thus, in either case we have a sequence of $\mu$-measurable polynomials $p_n(z,\overline{z})$ that converges uniformly to $\varphi$ on $K$.

Due to \cite[12.24]{RU} we have  $\lim_{n\to\infty}\left\|p_n(\AAA_t, \AAA_t^*) - \varphi(\AAA_t)\right\|=0$ for every $t\in\WWW$, implying the equality $$\scal{\varphi(\AAt)f}{f}=\lim_{n\to\iii}\scal{p_n(\AAA_t, \AAA_t^*)f}{f},\qquad t\in\Omega,\,\, f\in\HH,$$ 
 proving proclaimed  weak$^*$ $\mu$-measurability of the o.\,v. function $\varphi(\AAA)$.
\end{proof}

The following two results resolve the issue of the weak$^*$ $\mu$-measurability of $\AAA^p$, for all positive exponents.

\begin{lemma}\label{tv2}
   Let $p>0$ and let $\AAA\colon\WWW\to\BH$ be  a weakly$^*$ $\mu$-measurable function, such that $\AAt\geqslant0$ for all $t\in\WWW$. Then, the o.\,v. function $\AAA^p$ is weakly$^*$ $\mu$-measurable. 
\end{lemma}

\begin{proof}
    Define o.\,v. function $\BBB:\Omega\to\BH$ by the following formula
    $$\BBB_t=\begin{cases}
      \frac{1}{\|\AAA_t\|}\cdot \AAA_t, & \text{if}\ \|\AAA_t\| \neq 0, \\
      0, & \text{if}\,\,\|\AAA_t\|=0.
    \end{cases}$$ 
First, the mapping $\Omega\ni t\mapsto\|\AAA_t\|$ is $\mu$-measurable since for every $t\in\WWW$ holds
$\|\AAA_t\|=\sup\{|\langle\AAA_tf,g\rangle|:f,g\in S\},$
where $S$ denotes a countable dense subset of the unit sphere of $\HH$. Hence,
 $\BBB$ is weakly$^*$ $\mu$-measurable and $\|\BBB_t\|\in\{0,1\}$ for all $t\in\Omega$. Thus, $\sigma(\BBB_t)\subset[0,1]$, as $0\leqslant\BBB_t\leqslant I$ for all $t\in\WWW$. By applying Lemma \ref{tv1} to the function $\varphi(x):=x^p$ defined on the compact set $K:=[0,1]$, as $\varphi(0)=0$, we get the weak$^*$ $\mu$-measurability of $\BBB^p$. Finally, as $\AAA_t^p=\|\AAA_t\|^p\cdot \BBB_t^p$ 
for all $t\in\Omega$, we obtain the proclaimed weak$^*$ $\mu$-measurability of $\AAA^p$.
\end{proof}

\begin{theorem}\label{tv3}
Let $p>0$ and $\AAA\colon\WWW\to\BH$ be weakly$^*$ $\mu$-measurable. Then, the o.\,v. function $|\AAA|^p$ is weakly$^*$ $\mu$-measurable. 
\end{theorem}
\begin{proof}
As noted before, the o.\,v. function $\AAA^*\AAA=|\AAA|^2$ is weakly$^*$ $\mu$-measurable and $|\AAA_t|^2\geqslant0$ for all $t\in\WWW$. We get the final conclusion by applying Lemma \ref{tv2} to the o.\,v. function $|\AAA|^2$ and exponent $\frac{p}{2}>0$.
\end{proof}

We also need some well-known inequalities for positive operators. 
Let $0\leqslant A \in \mathcalb{B}(\mathcal{H})$  and let $f$ be a unit vector in $\mathcal{H}$. Due to \cite[Lemma 2.1]{McCarthy}, we have $\scal{Af}{f}^p\leqslant\scal{A^pf}{f}$ for $p\geqslant1$, and $\scal{A^pf}{f}\leqslant\scal{Af}{f}^p$ for $0<p\leqslant1$.
As a direct consequence, for every positive $A\in\BH$ and unit vector $f\in\HH$ holds
\begin{equation}\label{nejednakostMece od2}
\scal{A^pf}{f}=\scal{(A^2)^\frac{p}{2}f}{f}\leqslant\scal{A^2f}{f}^\frac{p}{2}=(\|Af\|^2)^{\frac{p}{2}}=\|Af\|^p,\,\,\,0<p\leqslant2.
\end{equation}
\begin{equation}\label{nejednakostVece od2}
\|Af\|^p=(\|Af\|^2)^{\frac{p}{2}}=\scal{A^2f}{f}^\frac{p}{2}\leqslant\scal{(A^2)^\frac{p}{2}f}{f}=\langle A^pf,f\rangle,\,\,\,p\geqslant2.
\end{equation}
Previous inequalities, along with Theorem \ref{tv3}, allow us to investigate o.\,v. functions of the form $\AAA^\alpha$. More precisely, we have the following result. 
\begin{theorem}\label{Gelpow}
Let $\AAA\colon\WWW\rightarrow\BH$ be weakly$^*$ integrable, such that $\AAA_t\geqslant0$ for every $t\in\WWW$. Then, for every $\alpha\in\left(0,\frac{1}{2}\right]$ we have $\AAA^\alpha\in L_s^{1/ \alpha}(\Omega,\mu,\BH)$ with a norm estimate $\|\AAA^\alpha\|_{L^{1/\alpha}_s}\leqslant\|\AAA\|_\GG^\alpha$.
\end{theorem}

\begin{proof}
    Weak$^*$ $\mu$-measurability of the o.\,v. function $\AAA^\alpha$ is provided by Theorem \ref{tv3}. As $\frac{1}{\alpha}\in[2,+\infty),$ due to  \eqref{nejednakostVece od2} we have
\begin{equation}\label{nejalfa}
\int_\WWW\|\AAA_t^\alpha f\|^\frac{1}{\alpha}\,d\mu(t)\leqslant\int_\WWW\scal{(\AAA_t^\alpha)^\frac{1}{\alpha}f}{f}d\mu(t)=\int_\Omega\scal{\AAA_tf}{f}d\mu(t)\leqslant\|\AAA\|_\GG,
\end{equation}
for every unit vector $f\in\HH$, implying that $\AAA^\alpha\in L_s^{1/ \alpha}(\Omega,\mu,\BH)$. By raising \eqref{nejalfa} to the power $\alpha$ and taking the supremum over the unit vectors in $\HH$ we get the proclaimed norm estimate.
    \end{proof}
Now, we are ready to define a new class of o.\,v. functions mentioned at the beginning of this section, which will play a central role in our main results.

\begin{definition}
    For $p>0$, we will denote by $L^p_\GG(\WWW,\mu,\BH)$ the class of all weakly$^*$ $\mu$-measurable   functions $\AAA:\Omega\to\BH$, such that $|\AAA|^p\in\GG(\WWW,\mu,\BH)$.
\end{definition}

In the case where $\Omega=\N$ with the counting measure $\mu$ we will use the simplified notation $\ell_\GG^p(\BH)$.

If $p=2$, the spaces $L^2_s(\WWW,\mu,\BH)$ and $L^2_\GG(\WWW,\mu,\BH)$ coincide. Integration of inequalities \eqref{nejednakostMece od2} and \eqref{nejednakostVece od2}
gives us the relations between these spaces for other exponents.
\begin{proposition}\label{Implikacije} For $p\in[1,2]$, we have  $L^p_s(\WWW,\mu,\BH)\subseteq L^p_\GG(\WWW,\mu,\BH)$, and for $p\geqslant2$ we have $L^p_\GG(\WWW,\mu,\BH)\subseteq L^p_s(\WWW,\mu,\BH)$.
\end{proposition}
As we will soon see, both of the inclusions in Proposition \ref{Implikacije} are generally strict for $p\neq2$. Also, as the following example shows, the spaces $\GG(\WWW,\mu,\BH)$ and $L_\GG^1(\WWW,\mu,\BH)$ are not comparable in general.

\begin{example}\label{GandL1G}
    Let $(e_n)_{n=1}^\iii$ be an o.\,n.\,b. of $\HH$. Define $\AAA,\BBB\colon\N\to\BH$ by $\AAA_n=e_1\otimes e_n^*$ and $\BBB_n=\frac{1}{n}e_n\otimes e_1^*$, for $n\in\N$. Then, $|\AAA_n|=e_n\otimes e_n^*$ and $|\BBB_n|=\frac{1}{n}e_1\otimes e_1^*$. Now, for every $f\in\HH$ we have   
    $$\sumn\scal{|\AAA_n|f}{f}=\sumn|\scal{f}{e_n}|^2=\|f\|^2<+\iii$$
    and $$\sumn|\scal{\BBB_nf}{f}|=\sumn\frac{1}{n}|\scal{e_n}{f}\scal{f}{e_1}|\leqslant|\scal{f}{e_1}|\cdot\sqrt{\frac{\pi^2}{6}}\cdot\|f\|<+\iii,$$
    implying  $\AAA\in\ell^1_\GG(\BH)$ and $\BBB\in\GG(\N,\BH)$. On the other hand, we have $\sumn\scal{|\BBB_n|e_1}{e_1}=\sumn\frac{1}{n}=+\iii$ and therefore $\BBB\notin\ell^1_\GG(\BH)$. Finally, for $h=\sumn\frac{1}{n}e_n\in\HH$ we have $$\sumn|\scal{\AAA_n h}{h}|=\sumn|\scal{e_1}{h}\scal{h}{e_n}|=\sumn|1\cdot\tfrac{1}{n}|=+\iii,$$
    proving that $\AAA\notin\GG(\N,\BH)$.
    $\hfill\triangle$
\end{example}

Since for $0<p\neq2$ there does not exist a constant $C_p$ such that for every unit  vector $f\in\HH$ holds $\scal{|A+B|^pf}{f}\leqslant C_p\cdot(\scal{|A|^pf}{f}+\scal{|B|^pf}{f})$ even when $\dim\HH=2$ (for $p=2$ the constant $C_2=2$, for the counterexample in the case $p=1$ see \cite{Lo76}), it is not immediately clear whether the spaces $L^p_\GG(\WWW,\mu,\BH)$ are closed under addition. That will generally not be the case (unless, for example, when $\dim\HH<+\iii$ or  $\WWW$ is a finite set), as the following two examples show. First, we deal with the case $0<p<2$.

\begin{example}
Let $(e_n)_{n=1}^\infty$ be an orthonormal basis of $\HH$ and let $0< p < 2$.
Set $s_n=\sin\frac{1}{\sqrt[p]{n}}$ and $c_n=\cos\frac{1}{\sqrt[p]{n}}$ and define unit vectors
$u_n = s_n\cdot e_1 + c_n\cdot e_{n+1}\in\HH$ for all $n \in \N$. Now, define o.\,v. functions $\AAA,\BBB\colon\N\to\BH$ by
$$\AAA_n:=u_n\otimes u_n^*,\qquad \BBB_n:=-e_{n+1}\otimes e_{n+1}^*,\qquad n\in\N.$$ 
Then, we have $|\AAA_n|=|\AAA_n|^p=\AAA_n$ and $|\BBB_n|=|\BBB_n|^p=e_{n+1}\otimes e_{n+1}^*$. Since for every $n\in\N$ holds
$$\AAA_n = u_n \otimes u_n^* = s_n^2\cdot( e_1 \otimes e_1^*) + s_nc_n\cdot ( e_1 \otimes e_{n+1}^* + e_{n+1} \otimes e_1^*)
+ c_n^2\cdot( e_{n+1} \otimes e_{n+1}^*),$$
we get
\begin{eqnarray*}
\langle |\AAA_n|^p f, f \rangle & = &\langle (u_n \otimes u_n^*)f, f \rangle = \langle f, u_n \rangle \cdot 
\langle u_n, f \rangle=|\scal{f}{u_n}|^2 \\&=& | s_n\cdot \langle f, e_1 \rangle + c_n\cdot \langle f, e_{n+1} \rangle |^2\leqslant(s_n\cdot|\scal{f}{e_1}|+|\scal{f}{e_{n+1}}|)^2 \\&\leqslant&2\cdot s_n^2\cdot\|f\|^2+2\cdot|\scal{f}{e_{n+1}}|^2.
\end{eqnarray*}
As $0<p<2$, we have $s_n^2\sim\frac{1}{n^{2/p}}$ as $n\to\iii$, and thus $\sumn s_n^2<+\iii$. By summing the previous estimate, due to Bessel's inequality, we get
\begin{equation*}
    \sumn\scal{|\AAA_n|^pf}{f}\leqslant\
    2\cdot\|f\|^2\cdot\sumn s_n^2+2\cdot\|f\|^2<+\infty,
\end{equation*}
for every $f\in\HH$, implying $\AAA\in \ell^p_\GG(\BH)$. Next, for every $f\in\HH$ holds
\begin{equation*}
    \sumn\scal{|\BBB_n|^pf}{f}=\sumn\scal{(e_{n+1}\otimes e^*_{n+1})f}{f}=\sumn|\scal{f}{e_{n+1}}|^2\leqslant\|f\|^2<+\iii,
\end{equation*}
providing  $\BBB\in \ell^p_{\GG}(\BH)$.
Finally, as
$$\AAA_n + \BBB_n = s_n^2\cdot (e_1 \otimes e_1^*) + s_nc_n\cdot ( e_1 \otimes e_{n+1}^* + e_{n+1} \otimes e_1^*)
- s_n^2\cdot (e_{n+1} \otimes e_{n+1}^*),$$
the direct computation gives us
\begin{equation*}
    |\AAA_n + \BBB_n|^2  = s_n^2\cdot(e_1\otimes e_1^*)+s_n^2 \cdot (e_{n+1} \otimes e_{n+1}^*),    
\end{equation*}
implying $|\AAA_n + \BBB_n|^p =s_n^p\cdot (e_1\otimes e_1^*)+s_n^p\cdot(e_{n+1}\otimes e_{n+1}^*)$ for every $n\in\N$. Thus, we have
$$\sum_{n=1}^\infty \langle |\AAA_n+\BBB_n|^pe_1,e_1\rangle=\sum_{n=1}^\infty s_n^p=\sum_{n=1}^\infty\sin^p\frac{1}{\sqrt[p]{n}}=+\infty.$$
Therefore, $\AAA+\BBB\notin \ell^p_\GG(\BH)$, proving that $\ell^p_\GG(\BH)$ is not a vector space.  $\hfill\triangle$
\end{example}
\begin{example}
    Let $(e_n)_{n=0}^\iii$ be an orthonormal basis in $\HH$ and let $p>2$. Define $\AAA,\BBB\colon\N\rightarrow\BH$ by $\AAA_n=e_n\otimes e_n^*$ and $\BBB_n=\frac{1}{\sqrt{n}}e_n\otimes e_0^*$, for $n\in\N$. Then, $$|\AAA_n|^p=\AAA_n=e_n\otimes e_n^*,\qquad \left\||\BBB_n|^p\right\|=\|\BBB_n\|^p=\frac{1}{n^{p/2}},\qquad n\in\N.$$
    Hence, $|\AAA|^p$ is weakly$^*$ integrable (its weak$^*$ integral equals $I-e_0\otimes e_0^*$). Also, since $p>2$, we have  $|\BBB|^p\in\ell^1(\BH)\subseteq\GG(\N,\BH)$. Thus, we have $\AAA,\BBB\in\ell^p_\GG(\BH).$ 
Furthermore, for every $n\in\N$ holds
\begin{align*}
\AAA_n+\BBB_n
&=
e_n\otimes e_n^*+\frac1{\sqrt n}\,e_n\otimes e_0^* =
e_n\otimes\left(e_n+\frac1{\sqrt n}\,e_0\right)^*=e_n\otimes v_n^*,
\end{align*}
where $v_n=e_n+\frac1{\sqrt n}\,e_0$.
Hence, for all $n\in\N$ we have
$|\AAA_n+\BBB_n|^2 = v_n\otimes v_n^*$ and $\|v_n\|^2=\|e_n+\tfrac{1}{\sqrt{n}}e_0\|^2=\|e_n\|^2+\tfrac{1}{n}\|e_0\|^2=1+\tfrac{1}{n}$.
Since the operator $v_n\otimes v_n^*$ is a rank-one and positive, we have
$$
|\AAA_n+\BBB_n|^p=(v_n\otimes v_n^*)^{p/2}=(\|v_n\|^2)^{p/2}\left(\frac{v_n}{\|v_n\|}\otimes \left(\frac{v_n}{\|v_n\|}\right)^*\right)^{p/2}=
\|v_n\|^{p-2}(v_n\otimes v_n^*).
$$
Finally, as $p>2$ and $\scal{v_n}{e_0}=\frac{1}{\sqrt{n}}$ for all $n\in\N$, we get
$$\sum_{n=1}^\infty\scal{|\AAA_n+\BBB_n|^pe_0}{e_0}=\sumn\|v_n\|^{p-2}|\scal{v_n}{e_0}|^2=\sum_{n=1}^\infty\left(1+\tfrac{1}{n}\right)^{\frac{p}{2}-1}\cdot\tfrac{1}{n}=+\iii,$$
thus $\AAA+\BBB\notin\ell^p_\GG(\BH)$, implying that $\ell^p_\GG(\BH)$ is not a vector space. \hfill$\triangle$
\end{example}

We note that since $L^p_\GG(\WWW,\mu,\BH)$ are not vector spaces in general for $p\neq2$, unlike $L^p_s(\WWW,\mu,\BH)$, the inclusions in Proposition \ref{Implikacije} can be strict.

\section{H\"{o}lder-type inequalities for operator-valued functions}

Operator-valued versions of the Cauchy-Schwarz inequality have attracted the attention of researchers during the last two decades. We refer to \cite{J05, JKL20} as examples of this line of research. 

We start with the following basic result, connecting various introduced norms of o.\,v. functions.

\begin{proposition}\label{OsnovniHelder}
Let $p,q> 1$ such that $\frac{1}{p}+\frac{1}{q}=1$ and let $\BBB\in L_s^q(\Omega,\mu,\BH)$.

\noindent \textup{\textbf{(a)}} If $\AAA\in L_s^p(\Omega,\mu,\BH)$ then $\AAA^*\BBB\in \GG(\Omega,\mu,\BH)$ and we have $$\|\AAA^*\BBB\|_\GG\leqslant\|\AAA\|_{L_s^p}\cdot\|\BBB\|_{L_s^q}.$$

\noindent \textbf{\textup{(b)}} If $\AAA\in L^p(\WWW,\mu,\BH)$ then $\AAA\BBB\in L^1_s(\WWW,\mu,\BH)$ and we have
$$\|\AAA\BBB\|_{L^1_s}\leqslant\|\AAA\|_{L^p}\cdot\|\BBB\|_{L_s^q}.$$
\end{proposition}
\begin{proof}

\noindent\textbf{\textup(a)}
Let $f,g\in\HH$ be unit vectors. By the scalar Hölder inequality we get
\begin{multline*}
\int_\Omega|\scal{\AAA_t^*\BBB_tf}{g}|\,d\mu(t)=\int_\Omega|\scal{\BBB_tf}{\AAA_t g}|\,d\mu(t)\leqslant\int_\Omega\|\AAA_tg\|\cdot\|\BBB_tf\|\,d\mu(t)\\
\leqslant\sqrt[p]{\int_\Omega\|\AAA_tg\|^p\,d\mu(t)}\cdot\sqrt[q]{\int_\Omega\|\BBB_tf\|^q\,d\mu(t)}\leqslant \|\AAA\|_{L_s^p}\cdot\|\BBB\|_{L_s^q}.
\end{multline*}
By taking the supremum over all unit vectors $f,g\in\HH$ we get the desired inequality.
\smallskip

\noindent \textbf{\textup{(b)}} Let $f\in\HH$ be an arbitrary unit vector. Then
\begin{multline*}
\int_\Omega\|\AAA_t \BBB_tf\|\,d\mu(t)\leqslant\int_\Omega\|\AAA_t\|\cdot\|\BBB_tf\|\,d\mu(t)\\\leqslant\sqrt[p]{\int_\Omega\|\AAt\|^p\,d\mu(t)}\cdot\sqrt[q]{\int_\Omega\|\BBB_tf\|^q\,d\mu(t)}\leqslant\|\AAA\|_{L^p}\cdot\|\BBB\|_{L_s^q}.\end{multline*}
By taking the supremum over $\|f\|=1$ we get the proclaimed estimate.
\end{proof}

As measurability of $\WWW\ni t\to\|\AAt\|\in\mathbb{R}$ can be deduced as in Lemma \ref{tv2}, we note that part \textup{\textbf{(b)}} of the above proposition remains valid under the weaker assumptions of weak$^*$ measurability of $\AAA$ and $\int_\WWW\|\AAt\|^p\,d\mu(t) < +\iii$. Moreover, the following example shows that reversing the order of $\AAA$ and $\BBB$ in that part causes the inequality to fail.

\begin{example} We will use the same o.\,v. functions $\AAA,\BBB:\N\to\BH$ as in Example \ref{GandL1G}.
Let $\AAA_n=e_1\otimes e_n^*$ and $\BBB_n=\frac{1}{n}e_n\otimes e_1^*$ for all $n\in\N$. Then, due to Bessel's inequality we get
$$\sum_{n=1}^{\infty}\|\AAA_n f\|^2=\sum_{n=1}^\infty|\scal{f}{e_n}|^2\cdot\|e_1\|^2\leqslant\|f\|^2<+\infty,\qquad f\in\HH,$$implying $\AAA\in \ell_s^2(\BH)$. Also, $\sum_{n=1}^\infty\|\BBB_n\|^2=\sum_{n=1}^\infty\frac{1}{n^2}<+\infty$, thus $\BBB\in\ L^2(\N,\BH)$. 
On the other hand, we have $$\sum_{n=1}^\infty\|\AAA_n\BBB_ne_1\|=\sum_{n=1}^\infty\left\|\frac{1}{n}(e_1\otimes e_1^*)e_1\right\|=\sum_{n=1}^\infty\frac{1}{n}\,\|\scal{e_1}{e_1}e_1\| = \sumn\frac{1}{n} =+\infty,$$
hence $\AAA\BBB\notin\ell_s^1(\BH)$.
$\hfill\triangle$
\end{example}

The following theorem presents a noncommutative Hölder-type inequality for $\BH$-valued functions. It can be seen as an extension of \cite[Lemma~2.1, (a1)]{JKL20} for more general exponents.

\begin{theorem}\label{OPHelder}
Let $1<q\leqslant 2\leqslant p<+\infty$ satisfy $\frac{1}{p}+\frac{1}{q}=1$ and assume that $\AAA\in L_s^q(\Omega,\mu,\BH)$ and $\BBB\in L^p_\GG(\Omega,\mu,\BH)$. Then $\AAA^*\BBB\in \GG(\Omega,\mu,\BH)$ and for all unit vectors $f\in\HH$ and $E\in\M$ holds
\begin{equation}\label{HelderNeKomut}
\begin{split}\!\!\scal{\left|\int_E\!\AAA^*\BBB\,d\mu \right|f}{f}&\!\leqslant\!\left\|\left(\int_E\!\AAA^*\BBB\,d\mu\right)\!\! f\right\|\!\leqslant\!\|
\AAA\|_{L^q_s}\sqrt[p]{\scal{\left(\int_E|\BBB|^p\,d\mu\!\right)\!f}{f}}.
\end{split}\end{equation}
\end{theorem}
\begin{proof} The first inequality in \eqref{HelderNeKomut} is immediate. Since $p\geqslant2$, using Proposition \ref{Implikacije} we get $\BBB\in L_s^p(\Omega,\mu,\BH)$. It follows that $\AAA^*\BBB\in \GG(\Omega,\mu,\BH)$, due to Proposition \ref{OsnovniHelder}. 
Let $f,g\in \HH$ such that $\|f\|=1$. Then, for every $E\in\M$ we have  
\begin{multline}\label{racun}
\left|\scal{\left(\int_E\AAA^*\BBB\,d\mu\right)f}{g}\right|=\left|\int_E \scal{\AAA_t^*\BBB_t f}{g}d\mu(t)\right|\leqslant\int_E|\scal{\AAA^*_t\BBB_tf}{g}|\,d\mu(t)\\
\leqslant\sqrt[\!q]{\int_E\|\AAA_tg\|^q\,d\mu(t)}\sqrt[p]{\!\int_E\|\BBB_tf\|^p\,d\mu(t)}
\leqslant\|\AAA\|_{L^q_s}\|g\|\sqrt[p]{\int_E\!\!\scal{|\BBB_t|^p f}{f}d\mu(t)},
\end{multline}
where the last inequality in \eqref{racun} follows from \eqref{nejednakostVece od2}.
Finally, by choosing the vector 
$g=\left(\int_E\AAA^*\BBB\,d\mu\right)f\in\HH$
in \eqref{racun}, we get 
$$\quad\quad\quad\left\|\left(\int_E\AAA^*\BBB\,d\mu\right)  f\right\|^2\leqslant\|\AAA\|_{L^q_s}\cdot\left\|\left(\int_E\AAA^*\BBB\,d\mu\right)  f\right\|\cdot\sqrt[p]{\int_E\scal{|\BBB_t|^p f}{f}d\mu(t)},$$
implying the second inequality in \eqref{HelderNeKomut}. 
\end{proof}
As an immediate consequence of the preceding theorem, we obtain the following Jensen-type inequality.
\begin{corollary}\label{PJensen}
Let $\BBB\in L^p_\GG(\WWW,\mu,\BH)$ for some $p\geqslant2$. Then, for every unit vector $f\in\HH$ and $E\in\M$ of finite measure holds
\begin{equation}\label{Jensen}
    \scal{\left|\int_E\BBB\,d\mu \right|f}{f}^p\leqslant\left\|\left(\int_E\BBB\,d\mu\right)f\right\|^p\leqslant\mu(E)^{p-1}\cdot\scal{\left(\int_E|\BBB|^p\,d\mu\right)f}{f}.
\end{equation}
\end{corollary}
\begin{proof}
    Let $\AAA=\chi_E\cdot I$ and $\frac{1}{p}+\frac{1}{q}=1$. Then, for every unit vector $f\in\HH$ holds
    $$\int_\WWW\|\AAA_tf\|^q\,d\mu(t)=\int_E\|f\|^q\,d\mu(t)=\mu(E)<+\iii,$$
    implying that $\AAA\in L^q_s(\WWW,\mu,\BH)$ and $\|\AAA\|_{L^q_s}=\mu(E)^\frac{1}{q}=\mu(E)^\frac{p-1}{p}$. Inequality \eqref{Jensen} now follows by applying Theorem \ref{OPHelder} to $\AAA$ and $\BBB$.
\end{proof}

Building on the ideas in \cite[Theorem~3.1, (a)]{JKL20} we derive a H\"older-type inequality for o.\,v. functions in the setting of $p$-modified s.\,n. ideals. It provides sufficient conditions under which the product family $\AAA^*\BBB$ induces a countably additive vector measure with values in the ideal $\ccphipn$.

\begin{theorem}\label{HelderT}
Let $\Phi$ be an s.\,n. function, let $1<q\leqslant 2\leqslant p<+\infty$ satisfy $\frac{1}{p}+\frac{1}{q}=1$ and let  $\AAA\in L_s^q(\Omega,\mu,\BH)$ and $\BBB\in L^p_\GG(\Omega,\mu,\BH)$ such that $\int_\Omega|\BBB|^p\,d\mu\in\ccphin$. Then, for every $E\in\M$ we have
$$\nu_{\AAA^*\BBB}(E)=\int_E\AAA^*\BBB\,d\mu\in\ccphipn$$
and the following inequality holds
\begin{equation}\label{heldermera}
\left\|\int_E\AAA^*\BBB\,d\mu \right\|_{\Phi^{(p)}} \leqslant \|\AAA\|_{L^q_s}\cdot\left\|\sqrt[p]{\int_E|\BBB|^p\,d\mu}\right\|_{\Phi^{(p)}}.    
\end{equation}
Furthermore, the induced vector measure $\nu_{\AAA^*\BBB}\colon\M\rightarrow\ccphipn$ is countably additive.
\end{theorem}
\begin{proof}
Let $e\in\BB$ and $E\in\M$ be arbitrary. Due to \eqref{HelderNeKomut}, for every $n\in\N$ holds
\begin{equation*}
\scal{\left|\int_E\AAA^*\BBB\, d\mu\right|e_n}{e_n}^p\leqslant\|\AAA\|^p_{L^q_s}\cdot\scal{\left(\int_E|\BBB|^pd\mu\right)e_n}{e_n}.
\end{equation*}
Furthermore, as $0\leqslant\int_E|\BBB|^p\,d\mu\leqslant\int_\WWW|\BBB|^p\,d\mu$, implying that $\int_E|\BBB|^p\,d\mu\in\ccphin$, and thus $\left(\scal{\left(\int_E|\BBB|^p\,d\mu\right)e_n}{e_n}\right)_{n=1}^\iii\in\cphin$, by
Proposition \ref{KriterijumPhiNula}. 
 Therefore, we have  $\left(\scal{\left|\int_E\AAA^*\BBB\, d\mu\right|e_n}{e_n}^p\right)_{n=1}^\iii\in\cphin$, implying $\left(\scal{\left|\int_E\AAA^*\BBB\, d\mu\right|e_n}{e_n}\right)_{n=1}^\iii\in\cphipn$ and thus $\left|\int_E\AAA^*\BBB\, d\mu\right|\in\ccphipn$, again due to Proposition \ref{KriterijumPhiNula}. Moreover, by the monotonicity of the s. n. function $\Phi$ we further get
\begin{multline*}
\Phi^{(p)}\left(\left(\scal{\left|\int_E\AAA^*\BBB\,d\mu \right|e_n}{e_n}\right)_{n=1}^\iii\right)^p=\Phi\left(\left(\scal{\left|\int_E\AAA^*\BBB\,d\mu \right|e_n}{e_n}^p\right)_{n=1}^\iii\right)\\
\leqslant\|\AAA\|^p_{L^q_s}\cdot\Phi\left(\left(\scal{\left(\int_E|\BBB|^p\,d\mu\right)e_n}{e_n}\right)_{n=1}^\iii\right)\leqslant\|\AAA\|^p_{L^q_s} \cdot \left\|\int_E|\BBB|^p\,d\mu\right\|_\Phi,
\end{multline*}
where the last inequality is based on Theorem \ref{Simon}. By taking $e$ to be the o.\,n.\,s. that consists of eigenvectors of the compact operator $\left|\int_E\AAA^*\BBB\,d\mu\right|$ we get
$$\left\|\int_E\AAA^*\BBB\,d\mu\right\|_{\Phi^{(p)}}^p \leqslant \|\AAA\|^p_{L^q_s} \cdot \left\|\int_E|\BBB|^p\,d\mu\right\|_\Phi=\|\AAA\|^p_{L^q_s}\cdot\left\|\sqrt[p]{\int_E|\BBB|^p\,d\mu}\right\|^p_{\Phi^{(p)}},$$
proving the inequality \eqref{heldermera}. Now, let $(E_n)_{n=1}^\iii$ be a sequence of mutually disjoint sets in $\M$. Due to \eqref{heldermera} we get 
\begin{equation*}
\left\|\nu_{\AAA^*\BBB}\left(\bigsqcup_{k=n+1}^{\iii} E_k\right)\right\|_\phip\leqslant\|\AAA\|_{L^q_s}\cdot\left\|\int_{\bigsqcup\limits_{k=n+1}^{\iii} E_k}\!\!|\BBB|^p\,d\mu \right\|_\Phi^{\frac{1}{p}}, \qquad n \in \N. 
\end{equation*}
Since the ideal $\ccphin$ is separable, the vector measure $\nuBp\colon\M\rightarrow\ccphin$, defined by the expression $\nuBp(E)=\int_E|\BBB|^p\,d\mu$ for $E\in\M$, is countably additive
 due to Theorem \ref{prbradit}. It follows that 
$$\left\|\int_{\bigsqcup\limits_{k=n+1}^{\iii} E_k}\!\!|\BBB|^p\,d\mu \right\|_\Phi=\left\|\nuBp\left(\bigsqcup\limits_{k=n+1}^{\iii} E_k\right)\right\|_\Phi=\left\|\sum_{k=n+1}^{\infty}\nuBp(E_k)\right\|_\Phi\rightarrow0$$
as $n\to\iii$, implying the proclaimed countable additivity of $\nu_{\AAA^*\BBB}.$
\qedhere
\end{proof}

\section{Pettis integrability of functions with values in separable s.\,n. ideals}

In this section we provide sufficient conditions for Pettis integrability of weakly$^*$ integrable o.\,v. functions with values in separable s.\,n. ideals. The motivation for studying Pettis integrability in this setting comes from the fact that separable s.\,n. ideals admit an explicit description of their dual spaces (see Theorem \ref{duali}). Our goal is to show that under suitable additional assumptions, the weak$^*$ integral $\int_E\AAA\,d\mu$ is in fact a $\ccphin$-Pettis integral. 

We first consider the ideal $\cci$ under the assumption that the measure $\mu$ is finite. The following theorem is a direct consequence of \cite[Proposition~3.2]{BJN} and Theorem \ref{TPetisA}. For the reader's convenience, we include a complete proof.

\begin{theorem}
Let $\mu(\WWW)<+\iii$ and $\AAA\in L_s^p(\Omega,\mu,\BH)$ for some $p>1$, such that $\AAA_t\in\cci$ for all $t\in\Omega$. Then, we have $\AAA\in\PP(\WWW,\mu,\cci)$.
\end{theorem}
\begin{proof}
Since $p>1$ and $\mu(\WWW)<+\iii$ we have $\AAA\in L^1_s(\WWW,\mu,\BH)$, thus $\AAA$ is weakly$^*$ integrable. Due to the H\"older inequality, for arbitrary $E\in\M$, we get the estimates
\begin{equation*}
\begin{split}
\|\nu_\AAA(E)&\|=\left\|\int_E\AAA\, d\mu\right\|=\sup_{\|f\|=1}\left\|\left(\int_E\AAA\, d\mu\right)f\right\|\leqslant\sup_{\|f\|=1}\int_E\|\AAA_t f\|\,d\mu(t)
\\&\leqslant\sup_{\|f\|=1}\left(\int_E\|\AAA f\|^p\,d\mu\right)^{\frac{1}{p}}\cdot\left(\int_E1^\frac{p}{p-1}\,d\mu\right)^\frac{p-1}{p}\leqslant\|\AAA\|_{L^p_s}\cdot\mu(E)^{1-\frac{1}{p}},
\end{split}
\end{equation*}
implying that the vector measure $\nu_\AAA$ is $\mu$-continuous. The final conclusion now follows from Theorem \ref{TPetisA}.
\end{proof}
Note that the proof of the previous theorem could also be derived from \cite[Corollary~1.2.38]{ABS}.

For every $0\leqslant A\in\cci$ and $\alpha>0$ we have $0\leqslant A^\alpha \in\cci$. Thus, due to 
Theorem \ref{Gelpow}, together with the previous theorem, in the case of finite measure $\mu$ and $\alpha\in\left(0,\frac{1}{2}\right]$, for every weakly$^*$ integrable function $\AAA$ such that $0\leqslant\AAA_t\in\cci$ we get that $\AAA^\alpha$ is strongly $\frac{1}{\alpha}$-integrable, and thus $\cci$-Pettis integrable. The following theorem provides Pettis integrability for the rest of the interval $(0,1)$

\begin{theorem}
    Let $\mu(\WWW)<+\infty$ and $\AAA\colon\WWW\to\BH$ be weakly$^*$ integrable o.\,v. function such that $0\leqslant\AAA_t\in\cci$ for every $t\in\WWW$. Then, for every $0<\alpha<1$ we have $\AAA^\alpha\in\PP(\WWW,\mu,\cci)$.
\end{theorem}

\begin{proof}
Let $\alpha \in (0,1)$ and $0\leqslant\AAA_t\in\cci$. Due to \cite[Lemma 2.1]{McCarthy}, for every unit $f\in\HH$ we have $\scal{\AAA_t^\alpha f}{f}\leqslant\scal{\AAA_t f}{f}^\alpha$ for every $t\in\WWW$. By integrating previous inequality over set $E\in\M$ and applying Holder inequality we further get
\begin{multline*}
\int_E \scal{\AAA_t^\alpha f}{f} d\mu(t) \leqslant
\int_E \scal{\AAA_t f}{f}^\alpha d\mu(t)
\leqslant \mu(E)^{1-\alpha}\cdot\left(\int_E \scal{\AAA_t f}{f} d\mu(t)\right)^\alpha\\
\leqslant  \mu(E)^{1-\alpha}\cdot \left\| \int_E \AAA\, d\mu\right\|^\alpha \leqslant \mu(E)^{1-\alpha}\cdot \left\| \int_\WWW \AAA\, d\mu\right\|^\alpha,
\end{multline*}
where the final inequality holds due to the positivity of $\AAA_t$. Thus, $\AAA^\alpha$ is weakly$^*$ integrable, so by taking the supremum over the unit vectors $f\in\HH$ we get
$$\|\nu_{\AAA^\alpha}(E)\|=\left\|\int_E \AAA^\alpha\, d\mu \right\| \leqslant \mu(E)^{1-\alpha}\cdot \left\| \int_\WWW \AAA\, d\mu\right\|^\alpha$$
and therefore the induced vector measure $\nu_{\AAA^\alpha}$ is $\mu$-continuous. Finally, the desired conclusion follows from  Theorem \ref{TPetisA}.
\end{proof}


Now, we turn our focus to o.\,v. functions with values in arbitrary separable s.\,n. ideals. In the sequel, we will not assume that the measure $\mu$ is finite. The following theorem extends the result of Theorem \ref{TPetisB} to all separable ideals of compact operators.
\begin{theorem}\label{Tpetispoz}
Let $\Phi$ be an s.\,n. function and  $\AAA\colon\WWW\rightarrow\BH$ be weakly$^*$ integrable function, such that $0\leqslant\AAA_t\in\ccphin$ for all $t\in\WWW$. Then, we have $\AAA\in\PP(\WWW,\mu,\ccphin)$ if and only if $\int_\WWW\AAA\, d\mu\in\ccphin$ in which case
\begin{equation}\label{petisnormapoz}
\|\AAA\|_\PP = \sup_{\substack{_{0\leqslant S\in\left(\ccphin\right)^*}\\ \|S\|_{\Phi^*}\leqslant 1}} \int_\WWW \tr(\AAt S)\, d\mu(t)= \left\|\int_\WWW\AAA\,d\mu\right\|_\Phi.
\end{equation}
Moreover, the supremum in \eqref{petisnormapoz} is attained.
\end{theorem}
\begin{proof}
If the function $\AAA$ is $\ccphin$-Pettis integrable, the conclusion follows directly from the Lemma \ref{petisjegeljfand}. Now, let $\int_\WWW\AAA\,d\mu\in\ccphin$. Due to Theorem \ref{prbradit}, the induced vector measure $\nu_\AAA$ is countably additive. Thus, $\AAA\in\PP(\WWW,\mu,\ccphin)$, by Theorem \ref{boljaVerzija}.
   
To prove the equality \eqref{petisnormapoz}, we first recall that $\|\int_\WWW\AAA\,d\mu\|_\Phi\leqslant\|\AAA\|_\PP$. For the proof of the converse inequality, let $X\in(\ccphin)^*$ be arbitrary, such that $\|X\|_{\Phi^*}\leqslant1$ and let $S=\frac{1}{2}(|X|+|X^*|)$. By the symmetry of $\Phi^*$ we have $\|X\|_{\Phi^*} = \||X|\|_{\Phi^*} = \||X^*|\|_{\Phi^*}$ and therefore $\|S\|_{\Phi^*}\leqslant\|X\|_{\Phi^*}\leqslant1$. Next, by the \eqref{tragpoz} for every $t\in\WWW$ we have
$$|\tr(\AAA_t X)|\leqslant \sqrt{\tr(\AAA_t |X|)}\cdot\sqrt{\tr(\AAA_t |X^*|)} \leqslant \frac{1}{2}(\tr(\AAA_t |X|) + \tr(\AAA_t |X^*|)) = \tr(\AAA_t S).$$
By integrating the previous inequality, we get
\begin{multline*}
\int_\WWW |\tr(\AAt X)|\, d\mu(t) \leqslant \int_\WWW \tr(\AAt S)\, d\mu(t) = \tr\left(\left(\int_\WWW\AAA\,d\mu\right) S\right)\\
\leqslant\left\|\int_\WWW\AAA\, d\mu\right\|_\Phi \cdot\|S\|_{\Phi^*} \leqslant\left\|\int_\WWW\AAA\, d\mu\right\|_\Phi,
\end{multline*}
where the first inequality also proves the first equality in \eqref{petisnormapoz}. Now, by taking the supremum over $X$ in the unit ball of $(\ccphin)^*$, we get that 
$\|\AAA\|_\PP \leqslant \left\|\int_\WWW\AAA\, d\mu\right\|_\Phi$, proving the equality \eqref{petisnormapoz}.

Finally, due to Hahn-Banach theorem, there exist $X\in (\ccphin)^*$ such that $\|X\|_{\Phi^*} = 1$ and $\int_\WWW \tr(\AAt X)\, d\mu(t) = \tr((\int_\WWW \AAA \,d\mu) X) = \|\int_\WWW \AAA\, d\mu\|_\Phi$. Similarly to the first part of the proof, if we take $S = \frac{1}{2}(|X|+|X^*|)$ we get $\|S\|_{\Phi^*} \leqslant 1$ and
$$\left\|\int_\WWW \AAA \,d\mu\right\|_\Phi \!=\! \int_\WWW\tr(\AAt X)\,d\mu \leqslant \int_\WWW|\tr(\AAt X)|\,d\mu \leqslant \!\int_\WWW\tr(\AAt S)\,d\mu \leqslant \left\|\int_\WWW \AAA \,d\mu\right\|_\Phi,$$
proving the final part of the theorem.
\end{proof}

\begin{remark}
In the case where $\Phi\sim\ell^1$, under the conditions of Theorem \ref{Tpetispoz} if $\int_\WWW\AAA\, d\mu\in\ccphi$ we get $\int_\WWW\|\AAA_t\|_\Phi \,d\mu(t)\leqslant\int_\WWW\|\AAA_t\|_1\,d\mu(t)<+\iii$ due to \cite[Lemma~2.1]{MMS}. Hence, as $\AAA$ is $\ccphi$-strongly $\mu$-measurable by Lemma \ref{glavnaPomoc}, it follows that $\AAA$ is $\ccphi$-Bochner integrable. Moreover, for the $\ell^1$ norm, again by \cite[Lemma~2.1]{MMS} we have  $\|\AAA\|_{\PP} = \int_\WWW\|\AAA\|_1\,d\mu = \|\int_\Omega \AAA \,d\mu \|_1$ which need not hold for other equivalent norms. Indeed, let $\AAA\colon\{1,2\}\to\BH$ be defined by $\AAA_1 = f\otimes f^*$ and $\AAA_2 = g\otimes g^*$ where $f,g\in\HH$ are orthogonal unit vectors, and let $\Phi=\frac{1}{2}(\ell^\infty +\ell^1)$. Then $\Phi\sim\ell^1$ while
$$\frac{3}{2} = \|\AAA_1 + \AAA_2\|_\Phi \neq \|\AAA_1\|_\Phi + \|\AAA_2\|_\Phi = 2.$$
\end{remark}

In the case where $\Phi\nsim\ell^1$ the Bochner integrability is not necessary, as the following example shows.

\begin{example}
    Let $\Phi\nsim\ell^1$ and let $a\in\cphin\setminus\ell^1$ (such an element exists as $\ell^1$ is the smallest sequence ideal and $\cphin\neq\ell^1$ due to \cite[Theorem~III.2.1.]{GK}) be a positive and decreasing sequence and define o.\,v. function $\AAA\colon\N\rightarrow\BH$ by the expression $\AAA_n=a_n\cdot(e_n\otimes e_n^*)$ for $n\in\N$. Then, we have $0\leqslant\AAA_n\in\KH\subset\ccphin$ for every $n\in\N$ and $$\int_\N\|\AAt\|_\Phi\, d\mu(t)=\sumn\|\AAA_n\|_\Phi=\sumn a_n=+\iii,$$ hence $\AAA$ is not $\ccphin$-Bochner integrable. On the other hand, we have $$\int_\N\AAA d\mu=\sumn a_n\cdot(e_n\otimes e_n^*)\in\ccphin,$$as the sequence $a=\left(s_n\left(\int_\N\AAA d\mu\right)\right)_{n=1}^\iii$ belongs to $\cphin$. $\hfill\triangle$
\end{example}

The positivity condition in Theorem \ref{HelderT} provides that $\int_E|\BBB|^p\,d\mu\in\ccphin$ for every $E\in\M$. However, this does not imply that $|\BBB_t|^p\in\ccphin$ for any $t\in\WWW$, as can be seen in \cite[Example~2.2]{MMS}. Under this additional assumption, we get the $\ccphipn$-Pettis integrability of $\AAA^*\BBB$ with its Pettis norm estimate.

\begin{theorem}\label{A*BPetis}
 Let $\Phi$ be an s.\,n. function, let $1<q\leqslant 2\leqslant p<+\infty$ satisfy $\frac{1}{p}+\frac{1}{q}=1$ and let  $\AAA\in L_s^q(\Omega,\mu,\BH)$ and $\BBB\in L^p_{\GG}(\Omega,\mu,\BH)$ such that $\BBB_t\in\ccphipn$ for every $t\in\WWW$ and   $\int_\Omega|\BBB|^p\,d\mu\in\ccphin$.    Then, $\AAA^*\BBB\in\PP(\WWW,\mu,\ccphipn)$ and we have the following estimate
\begin{equation}\label{petisNormaAB}
\|\AAA^*\BBB\|_{\PP\left(\WWW,\mu,\ccphipn\right)}\leqslant\|\AAA\|_{L^q_s}\cdot\left\|\,|\BBB|^p\,\right\|^\frac{1}{p}_{\PP\left(\WWW,\mu,\ccphin\right)}.
\end{equation}
\end{theorem}
\begin{proof}
We first note that $|\BBB|^p$ satisfies the conditions of Theorem \ref{Tpetispoz}. Thus, $|\BBB|^p\in\PP(\WWW,\mu,\ccphin)$ and $\||\BBB|^p\|_\PP=\left\|\int_\WWW|\BBB|^p\,d\mu\right\|_\Phi$.

Next, due to Theorem \ref{HelderT} we see that the measure $\nu_{\AAA^*\BBB}\colon\M\to\ccphipn$ given by $\nu_{\AAA^*\BBB}(E) = \int_E \AAA^*\BBB\, d\mu \in \ccphipn$ is countably additive. Also, from the assumptions of the theorem, we have $\AAA_t^*\BBB_t\in\ccphipn$ for every $t\in\WWW$. Hence, all the conditions of Theorem \ref{boljaVerzija} are satisfied and therefore $\AAA^*\BBB\in \PP(\WWW,\mu,\ccphipn)$.

To prove the norm estimate, let $X\in\ccphipz$ be such that $\|X\|_{\phip^*} = 1$. Next, we define a $\mu$-measurable function $\varphi\colon\WWW\to\C$ by
$$\varphi(t)=
\begin{cases}
\tfrac{\overline{\tr(\AAA_t^*\BBB_tX)}}
      {|\tr(\AAA_t^*\BBB_tX)|},
&\tr(\AAA_t^*\BBB_tX)\neq 0,\\[6pt]
0,
& \tr(\AAA_t^*\BBB_tX)=0.
\end{cases}
$$
As $|\varphi(t)| \leqslant 1$ for all $t\in\Omega$, the o.\,v. function $\varphi\cdot\AAA^*\BBB$ is also Pettis integrable due to \cite[Corollary~3.41]{Pet}. Again, due to Theorem
\ref{HelderT} we have
\begin{align*}
&\int_\WWW |\tr(\AAt^*\BBB_t X)|\,d\mu (t)= \int_\WWW \tr(\AAt^*\BBB_t X) \varphi(t)\, d\mu(t) = \int_\WWW \tr(\AAt^*\varphi(t)\BBB_t X)  \,d\mu(t)\\
=& \tr\left(\left(\int_\WWW\AAA^* (\varphi\BBB)\,d\mu\right) X\right) \leqslant\left\|\int_\WWW\AAA^*(\varphi \BBB)\,  d\mu \right\|_{\phip}
\!\!\!\!\leqslant \|\AAA\|_{L^q_s}\cdot\left\|\sqrt[p]{\int_\WWW|\varphi\BBB|^p\,d\mu}\right\|_{\Phi^{(p)}}\\ \leqslant& \|\AAA\|_{L^q_s}\cdot\left\|\sqrt[p]{\int_\WWW|\BBB|^p\,d\mu}\right\|_{\Phi^{(p)}}\!\!\!\!=\|\AAA\|_{L^q_s}\cdot\left\|\int_\WWW|\BBB|^p\,d\mu\right\|^\frac{1}{p}_\Phi\!\!\!=\|\AAA\|_{L^q_s}\cdot\||\BBB|^p\|^\frac{1}{p}_{\PP\left(\WWW,\mu,\ccphin\right)},
\end{align*}
where the first inequality in the third line holds since $$|\varphi(t) \BBB_t|^2 = \BBB_t^*\overline{\varphi(t)}\varphi(t)\BBB_t =|\varphi(t)|^2 \BBB_t^*\BBB_t \leqslant |\BBB_t|^2,\qquad t\in\Omega.$$ By taking the supremum over $\|X\|_{\phip^*} = 1$ we get the desired estimate \eqref{petisNormaAB}
\end{proof}
Finally, as a direct consequence of the previous theorem, we derive a Jensen-type norm inequality for separable s.\,n. ideals. 
\begin{corollary}
    Let $\Phi$ be an arbitrary s.\,n. function, and let $\BBB\colon\WWW\to\BH$ be weakly$^*$ $\mu$-measurable function such that $|\BBB|^p\in\PP(\WWW,\mu,\ccphin)$, where $\mu(\WWW)=1$ and  $2\leqslant p<+\infty$. Then, we have    $\BBB\in\PP(\WWW,\mu,\ccphipn)$ and the norm estimate
    \begin{equation}
        \left\|\BBB\right\|^p_{\PP\left(\WWW,\mu,\ccphipn\right)}\leqslant\||\BBB|^p\|_{\PP\left(\WWW,\mu,\ccphin\right)}.
    \end{equation}
\end{corollary}
\begin{proof}
    Let $\AAA\colon\WWW\to\BH$ be defined by $\AAt=I$  for all $t\in\WWW$ and let $q=\frac{p}{p-1}$. Since $\mu(\WWW)=1$, as in the proof of Corollary \ref{PJensen} we get $\AAA\in L^q_s(\WWW,\mu,\BH)$ and $\|\AAA\|_{L_s^q}=1$. Next, as $|\BBB|^p$ is $\ccphin$-Pettis integrable, we have  $\BBB_t\in\ccphipn$ for all $t\in\WWW$, as well as $\int_\WWW|\BBB|^p\,d\mu\in\ccphin$, due to Lemma \ref{petisjegeljfand}. By applying Theorem \ref{A*BPetis} we get that the function $\BBB=\AAA^*\BBB$ is $\ccphipn$-Pettis integrable and 
\begin{multline*}
\|\BBB\|_{\PP\left(\WWW,\mu,\ccphipn\right)}=\|\AAA^*\BBB\|_{\PP\left(\WWW,\mu,\ccphipn\right)}\\\leqslant\|\AAA\|_{L_s^q}\cdot\||\BBB|^p\|^\frac{1}{p}_{\PP\left(\WWW,\mu,\ccphin\right)}=\||\BBB|^p\|^\frac{1}{p}_{\PP\left(\WWW,\mu,\ccphin\right)},
\end{multline*}
completing the proof.
\end{proof}

\section*{Acknowledgements}
The authors of this research are partially supported by the Ministry of  Science, Technological Development and Innovation, Republic of Serbia through grant \texttt{451-03-33/2026-03/200104}.

\section*{Author contributions} All the authors wrote the manuscript text and reviewed the manuscript.

\section*{Data availability statement} No datasets were generated or analysed during the current study.

\section*{Declarations}
\subsection*{Competing interests} The authors declare no competing interests.

\end{document}